\documentclass[12pt]{amsart}
\usepackage{amsfonts,amssymb,latexsym,amsmath, amsxtra,verbatim}
\usepackage[all]{xy}
\usepackage[dvips]{graphics}

\usepackage[OT2,T1]{fontenc}
\DeclareSymbolFont{cyrletters}{OT2}{wncyr}{m}{n}
\DeclareMathSymbol{\Sha}{\mathalpha}{cyrletters}{"58}

\theoremstyle{plain}
\newtheorem{theorem}{Theorem}[section]

\newtheorem{conjecture}[theorem]{Conjecture}

\newtheorem{lemma}[theorem]{Lemma}
\newtheorem{proposition}[theorem]{Proposition}

\newtheorem{remark}[theorem]{Remark}
\theoremstyle{definition}

\newtheorem{definition}[theorem]{Definition}
\theoremstyle{remark}

\numberwithin{equation}{section}

\newcommand{\R}{\mathbb R}

\newcommand{\floor}[1]{\left\lfloor #1 \right\rfloor}

\def\P{ {\bf P}}

\def\({\left(}
\def\){\right)}
\def\<{\left<}
\def\>{\right>}

\newcommand{\wt}[1]{\widetilde{#1}}

\newcommand{\abs}[1]{\left|#1\right|}

\begin{document}

\title[Partitions with no $k$-sequence]{A proof of Andrews' conjecture on Partitions with no short sequences}
\author{Daniel M. Kane}
\author{Robert C. Rhoades}
\address{Stanford University, Department of Mathematics, Bldg 380, Stanford, CA 94305}
\email{dankane@math.stanford.edu}
\email{rhoades@math.stanford.edu}

\thanks{Research of the authors is supported by
NSF Mathematical Sciences Postdoctoral Fellowships.}

\date{\today}
\thispagestyle{empty} \vspace{.5cm}
\begin{abstract}
Holroyd, Liggett, and Romik introduced the following probability model.
Let  $C_1, C_2, \cdots$ be independent events with probabilities
$\P_s(C_n)= 1-e^{-ns}$
under a probability measure $\P_s$ with $0<s<1$.
%We can think of $C_n$ as the event that at least one occurs of a further set of $n$ independent events each of probability
%$e^{-s}$.
Let $A_k$ be the event
%$$A_k = \bigcap_{i=1}^\infty \( C_i \cup C_{i+1} \cup \cdots \cup C_{i+k-1}\)$$
that there is no sequence of $k$ consecutive $C_i$ that do not occur.
We given an asymptotic for $\P_s(A_k)$  with a relative error term that goes to $0$ as $s\to 0$.
This establishes a conjecture of Andrews.
\end{abstract}

%\subjclass[2000] {05A16 ,05A15, 11P82, 11P55}

\maketitle

\section{Introduction and Statement of Results}
Holroyd, Liggett, and Romik \cite{HLR}
introduced probability models
whose properties are useful to the study of two dimensional cellular automata and integer partitions.
Let $0<s<1$ and $C_1, C_2, \cdots$ be independent events with probabilities
$$\P_s(C_n):= 1-e^{-ns}$$
under a probability measure $\P_s$.
%We can think of $C_n$ as the event that at least one occurs of a further set of $n$ independent events each of probability
%$e^{-s}$.
Let $A_k$ be the event
$$A_k = \bigcap_{i=1}^\infty \( C_i \cup C_{i+1} \cup \cdots \cup C_{i+k-1}\)$$
that there is no sequence of $k$ consecutive $C_i$ values that do not occur.

%In application one would like an estimate for $\P_s(A_k)$ as a function of $s$.

Andrews \cite{AndPNAS} exhibited a connection between $\P_s(A_2)$ and one of
Ramanujan's mock theta functions $\chi(q)$. Later Andrews,  Eriksson, Petrov, and Romik \cite{AEPR}
explained further connections between this mock theta function and conditional probabilities in some probability spaces.
%This connection has been exploited to prove many strong results about the asymptotic behavior of the
%probability model \cite{AndPNAS}, bootstrap percolation models \cite{BM},
%and partitions without sequences \cite{BM1}.
No similar connections have been discovered for the other probability models.
%Using the connection with Ramanujan's mock theta functions,
Andrews \cite{AndPNAS}, using $q$-series identities, made the following conjecture.
\begin{conjecture}\label{conj:main} For each $k\ge 2$, there exists a positive constant $C_k$ such that
$$\P_s(A_k) \sim C_k s^{-\frac{1}{2}} \exp\( -\frac{\lambda_k}{s}\) \text{ as } s \downarrow 0 $$
with $\lambda_k := \frac{\pi^2}{3k(k+1)}.$
\end{conjecture}
Using the connection with Ramanujan's mock theta functions
Andrews \cite{AndPNAS}
proved the case $k=2$ with $C_2 = \sqrt{\frac{\pi}{2}}$.
Theorem 2 of Holroyd, Liggett, and Romik \cite{HLR} gives
$$ \log\( \P_s(A_k)\) \sim - \frac{\lambda_k}{s}$$
% = s^{O(1)}\exp\( -\frac{\pi^2}{3k(k+1)s}\)$$
for all $k$. This was later strengthened by  Bringmann and Mahlburg \cite{BM}, who showed that
$$
\exp\( - \frac{\lambda_k}{s}\) \ll_k \P_s(A_k) \ll_k s^{-\frac{2k-1}{2k}} \exp\( - \frac{\lambda_k}{s}\).
$$
We prove the following precise version of Andrews's Conjecture.
\begin{theorem}\label{thm:main}
Andrews's conjecture is true with $C_k = \frac{\sqrt{2\pi}}{k}$.  More specifically,
we have
$$\P_s(A_k) = \frac{\sqrt{2\pi}}{k} s^{-\frac{1}{2}} \exp\( - \frac{\pi^2}{3k(k+1)s}  + O_k\(s^{\frac{1}{2k+3}}\) \).$$
\end{theorem}
%{\bf obviously the error might not turn out to be this... but it is what i have so far.}
%\begin{remark}

%\end{remark}

%A slight strengthening of  this theorem (see Section \ref{sec:Idea})
This theorem is  applicable to enumerating  partitions with no $k$-sequences.
A partition $\lambda$ of $n$ has a $k$-sequence if
there are $k$ parts of consecutive sizes.  Partitions with no $k$-sequences, and further restrictions on the parts that may
occur, appear naturally in a number of partition problems.  Perhaps the first instance is in MacMahon's volume
\cite{macMahon}.  He interprets the combinatorial significance of the
Rogers-Ramanujan identity
\begin{equation}\label{eqn:rogersRamanujan}
\prod_{n=1}^\infty \frac{1}{(1-q^{5n-3})(1-q^{5n-2})} = \sum_{n=0}^\infty \frac{q^{n^2+n}}{(1-q)\cdots (1-q^n)}
\end{equation}
as saying that the number of partitions of $n$ into parts of the form $5n-3$ and $5n-2$ are equinumerous
with the partitions of $n$ into distinct parts with no $2$-sequences and no part of size $1$.
%See the works of Andrews \cite{And67a} and Andrews-Lewis \cite{AL} for many similar theorems.

Let $p_{k, r, >B}(n)$ be the number of partitions of $n$ with no $k$-sequence, no part occurring more than
$r$ times, and no parts of size $\le B$.  For simplicity, write $p_k(n) = p_{k, \infty, 0}(n)$.
Then \eqref{eqn:rogersRamanujan} is an identity for the generating function $\sum_{n=0}^\infty p_{2, 1, >1}(n)q^n$.
We have the following partition identities equating generating functions and infinite products
\begin{align*}
\sum_{n=0}^\infty p_{2,2,>1}(n) q^n =& \prod_{n=1}^\infty \frac{1}{(1-q^{6n-2})(1-q^{6n-3})(1-q^{6n-4})} \\
\sum_{n=0}^\infty p_{2,2, >0}(n)  q^n =& \prod_{n=1}^\infty \frac{(1-q^{6n-3})^2 (1-q^{6n})}{(1-q^{n})} \\
%\sum_{n=0}^\infty p_{2,1,>1}(n) q^n =& \prod_{n=1}^\infty \frac{1}{(1-q^{5n-2})(1-q^{5n-3})} \\
\sum_{n=0}^\infty p_{2,\infty,>1}(n) q^n =& \prod_{n=1}^\infty \frac{1}{(1-q^{6n})(1-q^{6n-2})(1-q^{6n-3})(1-q^{6n-4})}
\end{align*}
The first identity is due to Andrews \cite{And67a}, the second identity is due to MacMahon \cite{macMahon}
and the final identity is due to Andrews and Lewis \cite{AL}.  In each of these cases, modular techniques can be applied
to obtain exact formulas for the coefficients.

Moreover, we have
$$\sum_{n=0}^\infty p_2(n) q^n = \prod_{n=1}^\infty \frac{1+q^{3n}}{1-q^{2n}} \cdot \chi(q)$$
where $\chi(q) = \sum_{n=0}^\infty q^{n^2} \prod_{m=1}^n \frac{1+q^m}{1+q^{3m}}$
is one of Ramanujan's mock theta functions.
Bringmann and Mahlburg \cite{BM1}
use this connection with Ramanujan's mock theta function and an extension of the circle method
to prove a nearly exact formula for $p_2(n)$.

See the surveys of Ono \cite{onoSurvey} and \cite{onoCD} for more
applications of mock theta functions. Also, see the work of Knopfmacher and Munagi \cite{KM}
for similar constrained partition problems and connections with modular forms.

While there appear to be many connections  between partitions without sequences
and  modular and mock modular forms,
the general case appears out of reach of modular techniques.
The techniques of this paper can be applied to obtain asymptotics for $p_{k,r,B}(n)$
for any $k, r$ and $B$.
In particular, we have the following theorem for the asymptotic of $p_k(n)$.
%Let $p_k(n)$ be the number of partitions of $n$ which do not have a $k$-sequence
\begin{theorem}\label{thm:ksequencePartitions}
As $n\to \infty$ we have
$$p_k(n) \sim
\frac{1}{2 k} \( \frac{1}{6} \( 1- \frac{2}{k(k+1)}\)\)^{\frac{1}{4}} \frac{1}{n^{\frac{3}{4}}}
\exp\( \pi \sqrt{\frac{2}{3} \(1-\frac{2}{k(k+1)}\) n}\)
%\(1+ O_k\(n^{-\frac{1}{4k+6}}\)\)
.$$
\end{theorem}
%{\bf this is off by a factor of 2 in the case $k=2$}
%\begin{remark}
%In the case of $k=2$
%\end{remark}
%$p_2(n) =
%\( \frac{1}{4\sqrt{3} n^{\frac{3}{4}}} + \frac{1}{18 \sqrt{2} n} +  O\( \frac{1}{n^{\frac{5}{4}}}\) \) e^{\frac{2\pi}{3} \sqrt{n}}.$

%The approach used in this paper does not rely on any modular properties
%and can be applied to a large number of similar problems.

%Additionally, the techniques of this paper can be used to calculate asymptotics of partitions that locally ``avoid patterns''.
%For example,
%%it should be possible to use an approach similar to the one given here to give an asymptotic for the number
%%of partitions for which if there is a part of size $p$ then there are no parts of size $p+2$.
%Knopfmacher and Munagi \cite{KM}
%consider the problem of counting the number of partitions $\lambda = (\lambda_1, \cdots, \lambda_\ell)$ of $n$ such
%that there
%is no $j$ with $\lambda_j - \lambda_{j+1} = p$ for any fixed $p>0$.
%%A second type of problem arises from a cellular automata model introduced
%%by Bringmann, Mahlburg, and Mellit \cite{BMM}.
%%They are led to consider a $k$-sequence-type problem for overpartitions (see the paper
%%of Corteel and Lovejoy \cite{lovejoy} for  background on overpartitions).
%In other situations it is unclear whether the constants showing up in the asymptotic formulas will be explicitly computable.

In the next section, we sketch the approach taken to proving Theorem \ref{thm:main}.
\section{The Approach}\label{sec:Idea}
In this section, we sketch the proof  of Andrews's conjecture.
\subsection{Setup}
Denote the generating function for $p_k(n)$ by
$$G_k(q):= \sum_{n=0}^\infty p_k(n) q^n.$$
We let $q  = e^{-s}$. In Section 4 of  \cite{HLR} it is shown that
$$\P_s(A_k) = \frac{G_k( q)}{G(q)}$$
where $G(q) = \sum_{n=0}^\infty p(n) q^n = \prod_{n=1}^\infty \frac{1}{1-q^n}$
and $p(n)$ is the number of partitions of $n$.
Since precise asymptotics of $G(q)$ are well known, namely
$$G(q) =  \frac{1}{\sqrt{2\pi}} s^{-\frac{1}{2}} \exp\(\frac{\pi^2}{6s}  - \frac{s}{24}  + O(s^N)\)$$
for any $N$,
determination of the asymptotics of $G_k(q)$ is equivalent to the determination of the asymptotic of $\P_s(A_k)$.
We prove the following theorem.
%which is slightly stronger than Andrews's conjecture, but is
%used in the proof of Theorem \ref{thm:ksequencePartitions} and implies Theorem \ref{thm:main}.
\begin{theorem}\label{thm:partitionAsymp}
For each $k\ge2 $ we have
$$G_k(e^{-s}) = \frac{1}{k} \exp\( \frac{\pi^2}{6s} \(1-\frac{2}{k(k+1)}\) + O_k\(s^{\frac{1}{2k+3}}\)\)$$
as $s\to 0$.
\end{theorem}
\begin{remark}
A slight modification of the arguments presented establish
Theorem \ref{thm:partitionAsymp} with an error that is $o(1)$
for non-real $s$ satisfying  $|\Im (s)| = o\( \Re(s)\)$.
\end{remark}

Numerical calculations lead to the following conjecture for real $s$.
\begin{conjecture} For $s$ real and $s\to 0$
\begin{equation*}
G_k(e^{-s}) = \frac{1}{k} \exp\( \frac{\pi^2}{6s} \(1-\frac{2}{k(k+1)}\) + \sqrt{\frac{2}{9\pi}} s^{\frac{1}{k}} + O\(s^{\frac{2}{k}}\)\)
\end{equation*}
\end{conjecture}
The results of Bringmann and Mahlburg \cite{BM1} prove this in the case $k=2$.
This conjecture would imply that for $k>2$ the generating function $G_k(q)$ is not a usual modular form.   Indeed,
if $G_k(q)$ is a half integral weight modular form or
mixed mock modular form, we would expect
an asymptotic expansion that contains only powers of $s^{\frac{1}{2}}$.
%See Brenner \cite{brenner} for discussion of using asymptotics to prove the nonexistence of identities.
%Therefore, the only case in which one might hope for an identity with modular or mock modular forms is the case $k=2$.

\subsection{Method of Computation}

We use a recursion to calculate the generating function $G_k(q)$.
%Define the following truncation of the generating function $G_k(q)$.
Let $G_{k,N}(q)$ be the generating function for the number of partitions of $n$
with parts $<N$ and no $k$-sequence.
For $i = 0, 1, 2, \cdots, k-1$ we define
\begin{equation}
\wt{v}_i^k(N) := \sum_{\begin{subarray}{c} \lambda \text{ with parts } \le N \\
\text{ no $k$ consecutive parts} \\ \lambda \text{ has parts of size $N, N-1, \cdots, N-i+1$} \\ \lambda  \text{ has no part of size $N-i$}\end{subarray}} q^{\abs{\lambda}}.\end{equation}
We have the following recursion
$$\( \begin{array}{c} \wt{v}_0^k(N) \\ \wt{v}_1^k(N) \\ \vdots \\  \wt{v}_{k-1}^k(N) \end{array}\)
 = \( \begin{array}{cccc} 1 & 1 & \cdots & 1 \\ z(N) & 0 & \cdots & 0 \\ 0 & z(N) & \cdots & 0 \\
 0 & & \vdots & 0 \\ 0 & \cdots & z(N) & 0 \end{array}\)
 \( \begin{array}{c} \wt{v}_0^k(N-1) \\ \wt{v}_1^k(N-1) \\ \vdots \\  \wt{v}_{k-1}^k(N-1) \end{array}\)
 $$
where $z = z(n) := \frac{q^n}{1-q^n}$.  For convenience set
\begin{equation}\label{eqn:mDefinition}
m(n):= \( \begin{array}{cccc} 1 & 1 & \cdots & 1 \\ z(n) & 0 & \cdots & 0 \\ 0 & z(n) & \cdots & 0 \\
 0 & & \vdots & 0 \\ 0 & \cdots & z(n) & 0 \end{array}\).
\end{equation}
Therefore, we have
$$G_{k,N}(q) = \wt{v}_0^k(N) = {\bf e}_1^T \prod_{n=1}^{N} m(n) {\bf e}_1$$
where ${\bf e}_1 = \(\begin{array}{c} 1 \\ 0 \\ \vdots \\ 0\end{array}\).$
So we have $G_k(q) = \lim_{N\to \infty} G_{k,N}(q)$.
%  and we want to approximate the matrix$\prod_{n=1}^\infty m(n)$.

 Our main idea for evaluating this quantity is as follows.  If the $m(n)$ were simultaneously diagonalizeable, the product would be easy to evaluate and $G_k(q)$ would be approximately equal to the product of the largest eigenvalues.  This is not the case, but fortunately, the eigenvectors of the $m(n)$ vary slowly with $n$.  We diagonalize each of the $m(n)$ in order to approximate the matrix product in question.  The main term in our approximation is
% will still be
equal to the product of the largest eigenvectors of the $m(n)$, but we also have a correction term due to the changes in eigenbasis.

We note that this technique is similar to the adiabatic approximation
in quantum mechanics (see, for example,  Chapter 10 of \cite{griffiths}).  In
each case, we are sequentially applying a sequence of slowly-varying
matrices to a given initial vector (though in the adiabatic process,
this is done continuously rather than discretely).  In each case, we
write our vectors in terms of the (slowly changing) eigenbasis.  The
final outcome is approximated by taking the product (or integral) of
the eigenvalues, with a correction term due to the change of basis
(known as Berry's phase in the case of quantum mechanics).  The
justifications for this approximation are different, for while the
adiabatic approximation holds due to cancelation of cross terms due
to rapid oscillation, in our case the approximation holds because the
contribution from the non-primary eigenvectors may be safely
neglected

For small $n$ the main eigenvector is not a good approximation for the contribution to the generating function.
In fact, the non-primary eigenvalues contribute to the asymptotic approximation.
We may interpret this on the level of partitions.
%The parts of a random partition of size $n$ are guided by two types of behavior.
%Specifically,
Fristedt's \cite{fristedt} proved that for large $n$ the smallest parts of a partition are independent, while the large parts
are related via a Markov process.  Roughly speaking, the eigenvalues of $m(n)$ encode the markovity. Therefore, we
use a direct approximation to analyze the small parts of a partition without $k$-sequences.

We begin with some preliminary calculations of the matrices $m(n)$ in Section \ref{sec:Diagonalization}.
In Section \ref{sec:Early},
we give a direct computation for the generating functions $\wt{v}_i^k(N)$ for $N$ of size
$s^{-\frac{1}{k+1} - \epsilon}$ .  In Section \ref{sec:Late}, we calculate the contribution to $G_k(q)$ from
$\prod_{n>N} m(n)$.
In Section \ref{sec:Eigenvalues}, we estimate the product over the largest eigenvalues.
In Section \ref{sec:NonReal},
we deduce Theorem \ref{thm:partitionAsymp} and thus Theorem \ref{thm:main}.
In Section \ref{sec:PartitionAsymp} we give the proof of Theorem \ref{thm:ksequencePartitions}.

%%%%%%%%%%%%%%%%%%
%%                                               	   %%
%%		DIAGONALIZATION	   %%
%%						   %%
%%%%%%%%%%%%%%%%%%
\section{Calculations on the Diagonalization of $m(n)$}\label{sec:Diagonalization}
In this section, we collect some results on the eigenvalues and diagonalization of the matrices $m(n)$.
In this section, $k$ is fixed and $s$ is assumed to be small.
%with $|\Im(s)|<\Re(s)^{3/2}$ and $\Re(s)>0$.
%Consequentially we often suppress the dependence on $s$.
Errors are often written in big-$O$ notation. In almost all
cases the constants depend on $k$. We often suppress this dependence inside of the proofs.

Observe that
the characteristic polynomial of $\frac{1}{z(n)} m(n)$ is
$$\lambda^k - z(n)^{-1} \(\lambda^{k-1} +\cdots +\lambda+ 1\).$$
We begin by proving some basic results about the sizes of the eigenvalues of this polynomial when $z$
is either very big or very small.

\begin{lemma}\label{lem:primaryEigenvalueBound}
For $z\in\R$, let $\lambda_i$ be the roots of
$\lambda^k - z^{-1} \(\lambda^{k-1} + \cdots \lambda+1\) = 0$.
Then for $z$ large,
\begin{align*}
\lambda_i(z) =& \omega_iz^{-1/k} \( 1 + \frac{\omega_i}{k}z^{-1/k} + O_k\( z^{-\frac{2}{k}}\)\)\\
\end{align*}
where the $\omega_i$ are the distinct $k^{th}$ roots of unity.  Furthermore, for $z$ small one root satisfies
$$
\lambda_i = z^{-1}(1+O_k(z)),
$$
and all other roots satisfy
$$
\lambda_i = \omega_i(1+O_k(z)),
$$
where the $\omega_i$ here are distinct $k^{th}$ roots of unity other than 1.
\end{lemma}
\begin{proof}
For the first statement, note that we only need to show this for
$z\gg 1$.
We claim that $p(\lambda)=\lambda^k - z^{-1}(\lambda^{k-1}+\ldots+1)$ has a root within
$O(z^{-2/k})$ of $z^{-1/k}\omega$ for every $k^{th}$ root of unity $\omega$.  This follows easily noting that $p(z^{-1/k}\omega) = O(z^{-(k+1)/k})$, $|p'(z^{-1/k}\omega)| = \Theta(z^{-(k-1)/k})$ and that $|p^{(\ell)}(z^{-1/k}\omega)| = O(z^{-(k-\ell)/k}).$
This gives $\lambda_i = \omega_i z^{-\frac{1}{k}}\( 1+ O(z^{-\frac{1}{k}})\)$. The stronger claim follows from
$$\lambda_i^{k} = z^{-1}\(1+\lambda_i + O\(z^{\frac{2}{k}}\)\).$$

For the later two claims, we note that it suffices to consider $z\ll 1$.  For the second claim we note that $|p(z^{-1})| = O(z^{-k+1})$, $|p'(z^{-1})|=\Theta(z^{-k+1})$ and $|p^{(\ell)}(z^{-1})| = O(z^{-k+\ell})$.  For the final claim, note that if $\omega$ is a root of $x^{k-1}+\ldots+1$ that $|p(\omega)|=O(1)$, $|p'(\omega)| = \Theta(z^{-1})$ and $|p^{(\ell)}(\omega)| = O(z^{-1})$.
\end{proof}

\begin{lemma}\label{lem:distinct}
For every positive real $z$,
the polynomial $\lambda^{k} - z^{-1} \( \lambda^{k-1} + \cdots + \lambda + 1\)$ has no repeated roots.
\end{lemma}
%\begin{remark}
%This lemma remains true with $s$ having sufficiently small argument.
%\end{remark}
%\begin{proof}[Proof of Lemma \ref{lem:distinct}]
%We begin by assuming that $s$ is real.
\begin{proof}
Note that if $\lambda$ is a double root of the characteristic polynomial then
it satisfies $x^{k+1} - (1+z^{-1}) x^k + z^{-1} = 0$ and it is a root of the derivative
$\((k+1) x - (1+z^{-1}) k\) x^{k-1}$. Since $x= 0$ is clearly not a solution we have
that the double root is $\lambda = k (1+ z^{-1})/(k+1)$.  On the other hand, it is clear from the form of the characteristic polynomial, that there is a unique, non-repeated positive real root.

%We now deal with the case of complex $s$.    Since the eigenvalues vary continuously in $ns$, this implies that as long as $ns$ lies in some open neighborhood of the real axis that the eigenvalues are distinct.  By Lemma \ref{lem:primaryEigenvalueBound}, we have that the eigenvalues are distinct whenever $\Re(ns)$ is either sufficiently large or sufficiently small.  Therefore, as long as the argument of $s$ is sufficiently small, the eigenvalues are distinct for all $n$.
\end{proof}

\begin{definition}
By Lemma \ref{lem:distinct}, the roots of $\lambda^{k} - z(n)^{-1} \( \lambda^{k-1} + \cdots + \lambda + 1\)$ are distinct for any $n$
and $s$.
% with sufficiently small argument.
Therefore, the eigenvalues can be analytically continued to functions of $n\in \R^+$.
By Lemma \ref{lem:primaryEigenvalueBound}, as $s\rightarrow 0$, the various eigenvalues are asymptotic to
$e^{ \frac{2\pi i j}{ k}} z^{-\frac1k}$.  We let $\lambda_j(n)$ denote the root whose analytic continuation is asymptotic to
$
e^{2\pi i \frac{(j-1)}{k}}z^{-\frac{1}{k}}.
$
%For real $s$,
Thus $\lambda_1(n)$ is the unique positive real root of this polynomial.
%It is also clear that $\lambda_1$ is the root mentioned in Lemma \ref{lem:primaryEigenvalueBound} that is $z^{-1}(1+O_k(|z|))$ for $|z|$ small.
We note that $\lambda_j(n)z(n)$ are the eigenvalues of $m(n)$ and we call $\lambda_1(n)z(n)$ the \emph{primary eigenvalue} of the matrix $m(n)$.
\end{definition}

Since there are no repeated roots of the characteristic polynomial
of $m(n)$ for each eigenvalue $z\lambda_j = z(n)\lambda_j(n)$ of $m(n)$
%with $\lambda^k - z^{-1}(\lambda^{k-1} + \cdots + \lambda+ 1) =0$
we have the eigenvector
%$\(\begin{array}{c} \lambda^{k-1} \\ \vdots \\ \lambda \\ 1 \end{array}\),$ or, rescaling
$V_n^j:= \(\begin{array}{c} 1 \\ \lambda_j^{-1} \\ \vdots  \\ \lambda_j^{-k+1} \end{array}\).$
So we have
\begin{equation}
m(n) = A(n)D(n) A(n)^{-1}
\end{equation}
with
\begin{equation}
D = D(n) = \(\begin{array}{cccc}
z\lambda_1 & 0 & \cdots &0 \\
0 & z\lambda_2 & \cdots & 0 \\
 & & \vdots & \\
 0 & 0 & \cdots & z\lambda_k
 \end{array}\)
\end{equation}
and
\begin{equation}
A = A(n) = \(\begin{array}{cccc}
 1 & 1 & \cdots & 1 \\
 \lambda_1^{-1} & \lambda_2^{-1}& \cdots & \lambda_k^{-1} \\
  & \vdots & \vdots & \\
  \lambda_1^{-k+1} & \lambda_2^{-k+1} & \cdots & \lambda_k^{-k+1}
  \end{array}\).
  \end{equation}
% where $\lambda_ j = \lambda_j(n)$.
% and $\abs{\lambda_1} \ge \abs{\lambda_2} \ge \cdots \ge \abs{\lambda_k}$.

Next we turn to the transition matrices $A(n+1)^{-1} A(n)$.
\begin{lemma}\label{lem:transitionmatrix}
Let $\lambda_i = \lambda_i(n+1)$ and $\mu_i  = \lambda_i(n)$, then
 $A(n+1) = \( \lambda_{j}^{1-i}\)_{i,j}$ and $A(n) = \( \mu_{j}^{1-i} \)_{i,j} $
 and
 \begin{equation}
 T(n)= (T(n)^{i,j})_{i,j} := A(n+1)^{-1} A(n) = \(
 \prod_{m \neq i} \( \frac{\mu_j-\lambda_m}{\lambda_i-\lambda_m} \cdot \frac{\lambda_i}{\mu_j}\)\)_{i,j}
 \end{equation}
 where $i$ indexes the row and $j$ indexes the column of $T(n)$.
\end{lemma}

%\textbf{Note: We should use some index other than $k$ in this product}
 \begin{proof}
Note
 $$
 (A(n+1)^{-1} A(n))^T = A(n)^T (A(n+1)^{-1})^T.
 $$
 Furthermore,
 $$
 A(n)^T \left( \begin{matrix} a_0 \\ a_1 \\ \vdots \\ a_{k-1} \end{matrix} \right)
   = \left( \begin{matrix} p(\mu_1^{-1}) \\ p(\mu_2^{-1}) \\ \vdots \\ p(\mu_{k-1}^{-1}) \end{matrix} \right)
 $$
 where $p(x) = a_0 + a_1 x + \ldots + a_{k-1} x^{k-1}$.  Similarly,
$$
 A(n+1)^T \left( \begin{matrix} a_0 \\ a_1 \\ \vdots \\ a_{k-1} \end{matrix} \right)
  = \left( \begin{matrix} p(\lambda_1^{-1}) \\ p(\lambda_2^{-1}) \\ \vdots \\ p(\lambda_{k-1}^{-1}) \end{matrix} \right).
 $$

Therefore, the $(i,j)$ entry of $A(n+1)^{-1}A(n)$ is
$$
{\bf e}_j^T A(n)^T (A(n+1)^{-1})^T {\bf e}_i
$$
where ${\bf e}_i$ is the vector with a $1$ in the $i$th position and zeroes in all others.
This, in turn, is the value at $\lambda_j^{-1}$ of unique degree $(k-1)$ polynomial $p(x)$ so that
$p(\lambda_\ell^{-1}) = \delta_{\ell,i}$.  Therefore,
$$
p(x) = \prod_{m \neq i} \frac{x-\lambda_m^{-1}}{\lambda_i^{-1}-\lambda_m^{-1}}.
$$
Thus the $(i,j)$ entry is
$$
p(\mu_j^{-1}) = \prod_{m \neq i} \frac{\mu_j^{-1}-\lambda_m^{-1}}{\lambda_i^{-1}-\lambda_m^{-1}}
= \prod_{m \neq i}\left(\left(\frac{\mu_j-\lambda_m}{\lambda_i-\lambda_m}\right)\left(\frac{\lambda_i}{\mu_j}\right)\right).
$$
 \end{proof}
We will require some lemmas when dealing with transition matrices.
% and the size of the entries.

\begin{lemma}\label{lem:bounded1}
If $\lambda_1, \cdots, \lambda_k$ are the roots of $\lambda^k - z^{-1}(\lambda^{k-1} + \cdots + \lambda+1)=0$
then we have
$$\abs{\lambda_i - \lambda_j} \gg_k \abs{\lambda_j}.$$
\end{lemma}
\begin{proof}
By Lemma \ref{lem:primaryEigenvalueBound}
for $|z|\gg 1$, the $\lambda_i$ are proportional to distinct $k^{th}$ roots of unity, and thus the result follows for $z>C$ for some constant $C$.

By Lemma \ref{lem:primaryEigenvalueBound}
for $|z|\ll 1$, all but $\lambda_1$, are near distinct $k^{th}$ roots of unity, and $\lambda_1$ is roughly $z^{-1}$.  Thus if $i=1$ or $j=1$, then $|\lambda_i-\lambda_j| \gg z^{-1} \gg \abs{\lambda_j}$.  Otherwise, $|\lambda_i-\lambda_j|\gg 1 \gg \abs{\lambda_j}$.
Thus the result holds for $z<c$ for some constant $c$.

For $c\leq z \leq C$, we note that
$
\frac{\lambda_j}{\lambda_i-\lambda_j}
$
is a continuous function of $z$, and thus has some absolute upper bound.  Thus the Lemma holds in this range as well.
\end{proof}

\begin{lemma}\label{lem:bound2} In the notation of Lemma \ref{lem:transitionmatrix},
for any $i$ and $n$ we have $|\mu_i-\lambda_i| = O_k(|\lambda_i|(s+n^{-1}))$. Moreover, we have
$$\frac{\partial}{\partial z} \lambda_1(z) \ll  \lambda_1 (z) \(1 + \frac{1}{z}\)
\ \text{ and  } \ \  \frac{\partial^2}{\partial z^2} \lambda_1(z) \ll  \lambda_1(z) \( 1+ \frac{1}{z}\)^2
.$$
\end{lemma}
\begin{proof}
The first result follows from the claim that
$$
\frac{\partial \log(\lambda_i(z))}{\partial z} = O(1 + z^{-1}).
$$
This follows from the above bounds on $\lambda_i$ and the identity
\begin{equation}\label{eqn:firstderivative}
\frac{\partial}{\partial z} \lambda_i(z)
= -\frac{z^{-2}(\lambda_i^{k-1}+\ldots+1)}{k\lambda_i^{k-1}-z^{-1}((k-1)\lambda_i^{k-2} + \ldots + 1)}.
\end{equation}
In particular, the above allows us to check our claim for $z\gg 1$ and for $z\ll 1$.
As in Lemma \ref{lem:bounded1}, the claim follows for intermediate $z$ by a compactness argument.
The bound on the second derivative follows similarly.
We note that by differentiating $\lambda_1^{k+1} - \lambda_1^{k} - z^{-1} \( \lambda_1^{k} - 1\) =0$
we have the identity
\begin{align}\label{eqn:secondderivative}
\( (k+1) \lambda_1^k - k \lambda_1^{k-1} - z^{-1} k \lambda_1^{k-1}\) & \frac{\partial^2 \lambda_1}{\partial z^2} \nonumber
\\ =& 2z^{-3}( \lambda_1^{k}-1) - \frac{\partial \lambda_1}{\partial z} \cdot 2 z^{-2} k \lambda_1^{k-1}  \\
&- \(\frac{\partial \lambda_1}{\partial z}\)^2 \(  (k+1) k \lambda^{k-1} - k(k-1)(1+z^{-1}) \lambda_1^{k-2}\) \nonumber
\end{align}
\end{proof}

\begin{lemma}\label{lem:bound3}
In the notation of Lemma \ref{lem:transitionmatrix} for $j \ne m$
$$\abs{\frac{\mu_i - \lambda_m}{\lambda_j - \lambda_m} \cdot \frac{\lambda_j}{\mu_i}}$$
is bounded by some constant depending only on $k$.
\end{lemma}
\begin{proof}
This lemma follows from Lemmas \ref{lem:bounded1} and \ref{lem:bound2}. In particular,
in the case when neither $i$ nor $j$ is 1 then
$$
\abs{\lambda_j - \lambda_m} \gg \abs{\lambda_m} \gg \abs{\mu_i-\lambda_m}.
$$
Thus $\abs{\frac{\mu_i - \lambda_m}{\lambda_j - \lambda_m}}$ is bounded above as is $\abs{\frac{\lambda_j}{\mu_i}}$.

If $i=1$, the quantity in question is
$$
O\left(\abs{ \frac{\lambda_j}{\lambda_j-\lambda_m} } \right) = O(1).
$$ Similarly, the result follows for $j=1$.
\end{proof}

\begin{proposition}\label{prop:transitionmatrix}
The transition matrix $A(n+1)^{-1} A(n) = I_k+O_k\(s + \frac{1}{n}\)$
where $I_k$ is the $k\times k$ identity matrix.
\end{proposition}
\begin{proof}
We claim that
$$T(n)^{i,j} =  [A(n+1)^{-1} A(n)]_{j,i}  = \prod_{m\neq i} \frac{\mu_j - \lambda_m}{\lambda_i - \lambda_m}
\cdot \frac{\lambda_i}{\mu_j} = \delta_{i,j} + O(s +n^{-1}).$$
If $i\neq j$, by Lemma \ref{lem:bound2} the $m=j$ term of the product is
$$
\frac{\mu_j - \lambda_j}{\lambda_i - \lambda_j} \cdot
\frac{\lambda_i}{\mu_j} = O(s+n^{-1}) \cdot \frac{\lambda_i}{\lambda_i-\lambda_j} = O(s+n^{-1}).
$$
and the remaining terms are $O(1)$ by Lemma \ref{lem:bound3}.  This proves our bound for the off-diagonal coefficients.

For $i=j$, by Lemma \ref{lem:bound2} each $m$-term in the above product equals
$$
\frac{\lambda_i-\lambda_m + O(s+n^{-1})\abs{\lambda_i}}{\lambda_i-\lambda_m} = 1+O(s+n^{-1}).
$$
Taking a product over $m$ yields $1+O(s+n^{-1})$, which proves our claim.
\end{proof}

We conclude this section with one additional lemma dealing with the ratio of eigenvalues.
\begin{lemma}\label{lem:ratio}
If $i\ne 1$ and $ns \ll 1$ then
$$\frac{\abs{\lambda_i(n)}}{\abs{\lambda_1(n)}} \le  \exp\( - c (ns)^{\frac{1}{k}}\)$$
for some positive constant $c$.
\end{lemma}
\begin{proof}
This follows easily from the first case of Lemma \ref{lem:primaryEigenvalueBound}.
Namely,  for $i\neq 1$
$$
\frac{|\lambda_i|}{|\lambda_1|} = \exp(-\Omega(z^{\frac1k})) = \exp(-\Omega((ns)^{\frac1k})).
$$
\end{proof}

%%%%%%%%%%%%%%%%%%
%%                                               	   %%
%%  	         EARLY                           %%
%%						   %%
%%%%%%%%%%%%%%%%%%

 \section{Calculations of the early matrices}\label{sec:Early}
In this section, we construct an approximation for the vector $$\wt{V}(N):= \( \wt{v}_a(N)\)_{a=0}^{k-1} =
 \prod_{n=1}^N m(n) {\bf e}_1
% \(\begin{array}{c} 1 \\ 0 \\ \vdots \\ 0 \end{array}\)
 $$
with $s^{-1/2} \gg N \gg  s^{-\frac{1}{k+1}}\log(s^{-1})^{\frac{k}{k+1}} $.
%Throughout this section we will assume  $k\mid N$.  Again, we assume throughout that
%$\abs{\Im(s)} < \Re(s)^{\frac{3}{2}}$ with $\Re(s)>0$.

\begin{theorem}\label{thm:runup}
Assume that $k\mid N$ for some integer $N$ with $s^{-\frac{2}{k+2}}>N$ and $N$ greater than a sufficiently large multiple of $s^{-\frac{1}{k+1}}\log(s^{-1})^{\frac{k}{k+1}}$, then
\begin{align*}
\wt{v}_a(N) =&  \(sN\)^{-\frac{a}{k} - N\frac{k-1}{k} } e^{-\frac{N}{k}}\frac{1}{k^{\frac{3}{2}}}
\exp\( s^{\frac{1}{k}} N^{\frac{k+1}{k}} (k+1)^{-1} + O_k\(  sN^2 +s^{\frac{2}{k}}N^{\frac{k+2}{k}}
%+ N^{\frac{k+2}{k}}s^{\frac{2}{k}}
\) \)
\end{align*}
\end{theorem}

Before proving Theorem \ref{thm:runup} we introduce some notation.
Each entry of the vector is the generating function for the number of
 partitions with no $k$-sequence, no parts larger than $N$, and
 the largest missing part size is $-a\pmod{k}$.
 In this section we use the phrase ``run'' to refer to the gap between missing parts.
 Given a partition $\lambda$ with parts of size at most $N$ and no $k$-sequence, we let
 $$\ell = \ell(\lambda) = \sum_{ \text{``runs''}} (k  - \text{``length of run''}).$$
 It is clear that $\ell \le (k-1) N$.
 Note that the length of the run must be less than $k$ and that $\ell \equiv a \pmod{k}$.
 Let $n_j  = n_j(\lambda)$ be the parts not appearing in $\lambda$ satisfying
 $$0<n_1 < n_2 < \cdots < n_{\lfloor \frac{N+\ell}{k}\rfloor}.$$
 We have
 $$n_j = kj - \sum_{\text{``runs'' before } n_j} (k - \text{``length of run''}).$$
 We let $\{ t_j\}$ be the shortenings of the runs.
Namely, the length of the run before $n_i$ is equal to $$k- \abs{\{ j: t_j = i\}}$$
and we have
\begin{equation}\label{eqn:niFromti}
n_i = ki - \abs{\{j: t_j \le i\}}.
\end{equation}
 So we have
 $$0 \le t_1 \le t_2 \le \cdots \le t_\ell \le\left\lfloor \frac{N+\ell}{k} \right\rfloor.$$
 Note that a sequence of missing parts $\{ n_j\}$ determines the sequence $\{t_j\}$ and vice versa.
 We set $$M:= \left\lfloor \frac{N+\ell}{k} \right\rfloor = \frac{N}{k} + \frac{\ell - a}{k}.$$

So we have
\begin{equation}\label{eqn:runupentry}
\wt{v}_a(N):=
\prod_{n=1}^N z(n) \cdot  \sum_{\ell \equiv a\pmod{k} }\sum_{t_1 \le \cdots \le t_\ell} \prod_{i} z(n_i)^{-1},
\end{equation}
where the sum on $\ell$ runs over $\ell \le (k-1) N$.
For now we ignore the term $\prod_{n=1}^N z(n)$ as this term can be dealt with separately.
The idea for analyzing the remaining sum is that for $N$ about this size runs are likely to be of size $k-1$ or $k-2$.  One might interpret this
as saying that all the smallest parts want to appear subject to the constraint that every $k$th part cannot appear.
This agrees with Fristedt's probabilistic model.

Next we give a lemma which says we can ignore large $\ell$ values.
\begin{lemma}\label{lem:truncatedEll}
In the notation above,
$$
%\sum_{\begin{subarray}{c} \ell \equiv a\pmod{k}\\ 2kN^{(k+1)/k}s^{1/k} <  \ell \le (k-1) N\end{subarray} }
%\sum_{t_1 \le  \cdots \le  t_\ell} \prod_{i} z(n_i)^{-1}  = \exp\( - \Omega\( N^{\frac{k+1}{k}}s^{\frac{1}{k}}\)\)(sN)^{N/k}.
\sum_{\begin{subarray}{c} \ell \equiv a\pmod{k}\\ 2keN^{\frac{k+1}{k}}s^{\frac{1}{k}} <  \ell \le (k-1) N\end{subarray} }
\sum_{t_1 \le  \cdots \le  t_\ell} \prod_{i} z(n_i)^{-1}  = (sN)^{\frac{N}{k}} O(s^2).
$$
\end{lemma}
\begin{proof}
We note that
$$
\prod_i z(n_i)^{-1} \leq \prod_i z(N)^{-1} = z(N)^{-\floor{\frac{N+\ell}{k}}} \leq (sN)^{\frac{N+\ell}{k} -1}q^{O(N^2)}
\leq(sN)^{\frac{N}{k}}(sN)^{\frac{\ell}{k}}s^{-1}.
$$
The number of choices for $t$'s
is $\le \binom{N+\ell-1}{\ell}\le \binom{kN}{\ell}$. Thus $$
\sum_{t_1 \le \cdots \le t_\ell} \prod_i z(n_i)^{-1} = O\( s^{-1}\binom{kN}{\ell}  (sN)^\frac{N}{k} (Ns)^{\frac{\ell}{k}}\).$$
Noting that
$$
\binom{kN}{\ell} \leq \left( \frac{kNe}{\ell}\right)^\ell,
$$
this is at most
$$
O\left(s^{-1} (sN)^{\frac{N}{k}} \left(keN^{\frac{k+1}{k}}s^{\frac{1}{k}} \ell^{-1} \right)^\ell \right)
\leq O(s^{-1})(sN)^{\frac{N}{k} }  2^{-\ell}.
$$
We note that if $N$ is at least a sufficiently large multiple of $s^{-\frac{1}{k+1}}\log(s^{-1})^{\frac{k}{k+1}}$, then $2^\ell = O(s^3)$.
Summing on $\ell$, yields the result.
\end{proof}

\begin{proof}[Proof of Theorem \ref{thm:runup}]
We apply Lemma \ref{lem:truncatedEll} to the summation in \eqref{eqn:runupentry} and, unless otherwise stated,
in the remainder of this proof we assume the sum on $\ell$ is truncated by $\ell < 2keN^{\frac{k+1}{k}}s^{\frac{1}{k}}$ at a cost of a negligible error.

 We will use the following calculations throughout the proof.
 We have
 $z(n)^{-1} = \frac{1-q^n}{q^n}$, but $q^n = e^{-ns}$, so $1-q^n = ns\( 1+O\(ns\)\)$.  Moreover,
 $\prod q^{n_i} = e^{-\sum n_i s}$
 but $s \sum n_i \le N^2 s \ll 1$ by construction. Therefore we have
 $$\prod_{i} z(n_i)^{-1} = \prod n_i s (1+O(n_i\abs{s})) = s^{M} \prod_{i} n_i
 \cdot (1+O(sN^2)).$$

%Therefore, removing the ordering on the $t$'s we need to estimate
%$\sum_{\ell \equiv a\pmod{k}} \sum_{t_1, \cdots , t_\ell}
%\frac{s^M}{\ell!} \prod_i n_i \abs{\{ j\le i: t_j = t_i\} } (1+O(sN^2))$$
Recall that
\begin{equation}\label{eqn:ni_approximation}
n_i = ki - \abs{\{j: t_j \le i\}} = ki \exp\( - \frac{ \abs{\{j: t_j \le i\}} }{ki} + O\( \frac{\ell \abs{\{j: t_j \le i\}} }{i^2}\)\).\end{equation}
 So the sum becomes
 \begin{align*}
\sum_{\ell \equiv a\pmod{k}} & \sum_{t_1 \le \cdots \le t_\ell} \prod_{i} z(n_i)^{-1}  \\
 =  &\sum_{\ell \equiv a \pmod{k}} (sk)^M M! \sum_{t_1\le  \cdots \le  t_\ell} \prod_j
 \exp\( - \sum_{i \ge t_j} \frac{1}{ki}  +O\( \frac{\ell}{i^2}\)\)\(1+ O(sN^2)\) \\
 =& \sum_{\ell \equiv a \pmod{k}} (sk)^M \frac{M!}{\ell!}\sum_{t_1, \cdots, t_\ell}  \exp\( -\frac{1}{k}
 \sum_{j=1}^\ell \log\( \frac{M}{t_j}\) + O\(\frac{\ell}{t_j}\)\) \\ & \hspace{2.3in}
 \times\prod_j \(1 + \abs{ \{ i<j: t_i = t_j\}}\)\(1+ O(sN^2)\) \ \\
 &= \sum_{\ell \equiv a \pmod{k}} (sk)^M \frac{M! M^\ell}{\ell!} \(\int_0^1 t^{\frac{1}{k}} e^{O\(\frac{\ell}{Mt}\)} dt\)^\ell
 \( 1+ O\(\frac{\ell^2}{N} + sN^2\)\)  \\
 &= \sum_{\ell \equiv a \pmod{k}} (sk)^M \frac{M! M^\ell}{\ell!} \(\int_0^1 t^{\frac{1}{k}} \(1 + O\(\frac{\ell}{Mt}\)\) dt\)^\ell
 \( 1+ O\(\frac{\ell^2}{N} + sN^2\)\) \\
 &= \sum_{\ell \equiv a\pmod{k}} (sk)^M \frac{M! M^\ell}{\ell!} \( \frac{k}{k+1} \)^\ell\( 1+ O_k\(
% N^{\frac{k+2}{k}}s^{\frac{2}{k}}
s^{\frac{2}{k}}N^{\frac{k+2}{k}}
 + sN^2\)\), \end{align*}
 where we use that $\delta \le \frac{k+1}{k} \epsilon$.
  The third line is obtained by removing the ordering on the $t_i$'s.  The product
 $\frac{1}{\ell!} \prod_{j} \( 1+ \abs{\{ i <j : t_i = t_j\}}\)$ accounts for the introduced over-counting.  The fourth line is obtained by approximating the sum over $t_j$ (once $t_i$ has been fixed for $i<j$) of $t_j^{1/k}\(1 + \abs{ \{ i<j: t_i = t_j\}}\)$ by $\int t^{1/k}dt$.
Additionally, in the fifth line we note that term $O\(\frac{\ell}{Mt}\)$ is always negative, see \eqref{eqn:ni_approximation}.

Applying Stirling's approximation to $M!$, and suppressing the errors, we see that the above is equal to
\begin{align*}
&\(\frac{s}{e}(N-a)\)^{\frac{N-a}{k}}  \sqrt{2\pi \frac{N-a}{k}}
\sum_{\ell \equiv a \pmod{k}} \(\frac{s}{e}\)^{\frac{\ell}{k}}
\( \frac{N+\ell-a}{N-a}\)^{\frac{N-a}{k}+\frac{1}{2}}  \(N+\ell -a\)^{\ell\(\frac{k+1}{k}\)}  \frac{1}{\ell!}
\(\frac{1}{k+1}\)^\ell \\
&= \( \frac{s(N-a)}{e}\)^{\frac{N-a}{k}} \sqrt{2\pi \frac{N-a}{k}} \sum_{\ell \equiv a\pmod{k}}
\( \frac{1}{k+1}  s^{\frac{1}{k}} (N-a)^\frac{k+1}{k}\(1 + O\( \frac{\ell}{N}\)\)  \)^\ell
\frac{1}{\ell!}  \\
=&\( \frac{s(N-a)}{e}\)^{\frac{N-a}{k}} \sqrt{2\pi \frac{N-a}{k}}\( \sum_{\ell \equiv a\pmod{k}}
\( \frac{1}{k+1}  s^{\frac{1}{k}} (N-a)^\frac{k+1}{k}  \)^\ell
\frac{1}{\ell!} \)   \(1 + O\(
%N^{\frac{k+2}{k}}s^{\frac{2}{k}}
s^{\frac{2}{k}}N^{\frac{k+2}{k}}\)\)
\end{align*}
where we have used $\(\frac{N+\ell - a}{N-a}\)^{\frac{N-a}{k}} = \(1+\frac{\ell}{N}\)^{\frac{N-a}{k}}= e^{\frac{\ell}{k}}$
times a negligible error.

Extending the sum to a sum over all $\ell$ rather than those with $\ell < 2ke s^{\frac{1}{k}}N^{\frac{k+1}{k}}$ introduces a negligible error.
The completed sum over $\ell$ is the sum over every $k$-th term of an exponential.  Thus,
suppressing the above error terms, we have
\begin{align*}
\sum_{\ell \equiv a\pmod{k}} &
\(s^{\frac{1}{k}} (N-a)^\frac{k+1}{k} (k+1)^{-1} \)^\ell
\frac{1}{\ell!}
\\= & \frac{1}{k}  \sum_{t\pmod{k}} \zeta_k^{at}\exp\(s^{\frac{1}{k}} (N-a)^\frac{k+1}{k} (k+1)^{-1} \zeta_k^t \)  \\
=& \frac{1}{k}  \exp\( s^{\frac{1}{k}} (N-a)^\frac{k+1}{k} (k+1)^{-1}\)
\(1+O\left(\exp\left(- \frac{s^{\frac{1}{k}}N^{\frac{k+1}{k}}}{2k(k+1)} \right) \right)\)\\
%\(1+O\(\exp\( - \frac{1}{k^2(k+1)}  s^{\frac{1}{k}} N^\frac{k+1}{k}\)\)\) \\
=&  \frac{1}{k}  \exp\( s^{\frac{1}{k}} N^\frac{k+1}{k} (k+1)^{-1}\)
\(1+O\( (sN)^{\frac{1}{k}} \)\)
\end{align*}
where we have approximated $N-a$ by $N$.
%\begin{remark}
%For non-real $s$ with sufficiently small argument we note that we can safely ignore the terms with $t\neq 0$ since the argument of $s$ is (and thus of $s^{1/k}$) is arbitrarily small, thus all of the exponentials in question have the same norm, but the term with $t\neq 0$ is the only one where the input is nearly aligned along the real axis.
%Moreover, the individual errors in the terms of the above sum produce an error of size
%$$
%O\left(N^{2\delta-1} + |s|N^2 \right)\exp\left(|s|^{\frac{1}{k}}N^{\frac{k+1}{k}}(k+1)^{-1} \right).
%$$
%We note that $$|s|^{\frac{1}{k}}N^{\frac{k+1}{k}}(k+1)^{-1} = \Re\left(s^{\frac{1}{k}}N^{\frac{k+1}{k}}(k+1)^{-1} \right)+O(1).$$  This is because $\arg(s^{\frac{1}{k}}) = \frac{\arg(s)}{k} = O(|s|^{\frac12})$.
%On the other hand, $|s|^{\frac{1}{k}}N^{\frac{k+1}{k}}=O(|s|^{-\frac{k+1}{k} \epsilon}).$
%This leads to the constraint that
%$\Im(s)  = o\(\Re(s)\).$
%\end{remark}

To finish the proof of the theorem  we
use
$$\prod_{n=1}^N z(n) = \prod_{n=1}^N (sn)^{-1} \(1+O(ns)\) = \frac{s^{-N}}{N!} \(1+ O\( N^2 s\)\) =
\frac{e^{N}}{(sN)^N \sqrt{2\pi N}} \(1+ O\( N^2 s\)\).$$
\end{proof}

%\begin{remark}
%Note that the generating function for partitions with parts of size less than or equal to $N+\ell$ which does not use
%every $k$th part is
%$$\frac{\prod_{n=1}^{N+\ell} z(n)}{\prod_{n=1}^{\lfloor \frac{N+\ell}{k} \rfloor} z(kn)}$$
%arguing as above $\prod_{n=1}^{\lfloor \frac{N+\ell}{k} \rfloor } z(kn)^{-1}  \approx (sk)^M M!$ in our sum we have
%the correction factor $\frac{M^\ell}{\ell!} \( \frac{k}{k+1}\)^\ell$. The term $\frac{M^\ell}{\ell!}$ is the number of ways to choose shorter gaps than gaps of length $k$.
%The factor $\( \frac{k}{k+1}\)^\ell$ is the correction factor for the modification to missing parts.
%\end{remark}

Before concluding this section we give a comparison between $\wt{v_0}(N)$ and the eigenvectors of $m(N)$.
We let $V_n^i$ be the eigenvector $\(\begin{array}{cccc} 1 & \lambda_i(N)^{-1} & \cdots & \lambda_i(N)^{-k+1} \end{array}\)^T$
of $m(n)$ corresponding to the eigenvalue $\lambda_i(n)z(n)$.
% From the previous section we have that the primary eigenvalue is
% $z(n)\lambda_1(n)$ where
% $$\lambda_1(n)^k = z(n)^{-1} \(\lambda_1(n)^{k-1} + \cdots + \lambda_1(n) + 1\)$$
% and the corresponding eigenvector is $$\(\begin{array}{c} 1\\ \lambda_1^{-1} \\ \cdots \\ \lambda_1^{-k+1} \end{array}\).$$
% From Lemma \ref{lem:primaryEigenvalueBound} we see that
%$$V_N^1 = \( \begin{array}{c} (Ns)^{\frac{k-1}{k}} \\ (Ns)^{\frac{k-1}{k} - \frac{1}{k}} \\ \cdots \\ 1 \end{array}\) +
%O(sN)^{\frac{1}{k}} \(\begin{array}{c} 1 \\ \vdots \\ 1 \\ 0 \end{array}\).$$

\begin{proposition}\label{prop:runupInEigenbasis} In the notation above, with
$V_N^i = \(\begin{array}{cccc} 1 & \lambda_i(N)^{-1} & \cdots & \lambda_i(N)^{-k+1} \end{array}\)^T$ we have
\begin{align*}
\wt{V}(N) = (Ns)^{-N\frac{k-1}{k}} e^{-\frac{N}{k}}
\frac{1}{k^{\frac{3}{2}}} & \exp\( s^{\frac{1}{k}} N^{\frac{k+1}{k}} (k+1)^{-1}  + O\(sN^2 +
%N^{\frac{k+2}{k}}s^{\frac{2}{k}}
s^{\frac{2}{k}}N^{\frac{k+2}{k}}\)
\) V_N^1  \\& + \sum_{i>1} C_N^i V_N^i
\end{align*}
where $$C_N^i \ll (Ns)^{-N\frac{k-1}{k}} e^{-\frac{N}{k}}\exp\( s^{\frac{1}{k}} N^{\frac{k+1}{k}} (k+1)^{-1}
\)O\(sN^2 +
s^{\frac{2}{k}}N^{\frac{k+2}{k}}
%N^{\frac{k+2}{k}}s^{\frac{2}{k}}
\).$$
\end{proposition}
%Therefore, we see that up to errors
%\begin{equation}
%\wt{V}(N) =  (Ns)^{-2N\frac{k-1}{k}} e^{-\frac{N}{k}}
%\frac{1}{k^{\frac{3}{2}}} \exp\( s^{\frac{1}{k}} N^{\frac{k+1}{k}} (k+1)^{-1} \)
%V_N^1
%\end{equation}
\begin{proof}
Since the eigenvectors, form a basis, there exist $C_N^i$ so that $\wt{V}(N) = \sum_{i\geq 1} C_N^i V_N^i$.
Applying Theorem \ref{thm:runup}, we have that
$$\wt{v}_a(N) = \wt{v}_0(N) (sN)^{-\frac{a}{k}}\left( 1 + O\left(
%(sN)^{\frac{1}{k}}
 s^{\frac{2}{k}}N^{\frac{k+2}{k}}
+ sN^2 \right) \right).$$
By Lemma \ref{lem:primaryEigenvalueBound} we have that
$$
\lambda_j(N) = e^{2\pi i \frac{(j-1)}{k}} (sN)^{\frac1k} ( 1 + O((sN)^{\frac{2}{k}})).
$$
Therefore, we have that for $0\leq a \leq k-1$,
$$\wt{v}_0(N)\left( 1 + O\left(
%(sN)^{\frac{1}{k}}
s^{\frac{2}{k}}N^{\frac{k+2}{k}}
+ sN^2 \right) \right) = \sum_{i=1}^k e^{-\frac{2\pi i a (j-1)}{k}}
( 1 + O(sN)^{\frac{2}{k}})C_N^i.$$
In other words if $B$ is the matrix with $(a,j)$ entry $e^{-\frac{2\pi i a (j-1)}{k}}$, then $B+O(sN)^{\frac 2k}$ times the vector of $C_N^i$ equals a vector whose entries are $\wt{v}_0(N)\left( 1 + O\left(
s^{\frac{2}{k}}N^{\frac{k+2}{k}}
%(sN)^{\frac{1}{k}}
+ sN^2 \right) \right)$.
Noting that the inverse of $B+O(sN)^{\frac 2k}$ is $B^{-1}+O(sN)^{\frac 2k}$ this implies that
$C_N^1 = \wt{v}_0(N)\left( 1 + O\left( %(sN)^{\frac{1}{k}}
s^{\frac{2}{k}}N^{\frac{k+2}{k}}
+ sN^2 \right) \right)$, and $C_N^i = \wt{v}_0(N) O\left( %(sN)^{\frac{1}{k}}
s^{\frac{2}{k}}N^{\frac{k+2}{k}}
+ sN^2 \right) $ for $i>1$.  This proves our Proposition.
\end{proof}

Finally, the next proposition compares $\wt{v}_0(N)$ to the product of the eigenvalues.
\begin{proposition}\label{prop:transitionComparison} In the notation above,
\begin{equation*}\label{eqn:transitionComparison}
\frac{\wt{v}_0(N)}{\prod_{n=1}^N \lambda_1(n) z(n)} = \frac{1}{k^{\frac{3}{2}}(2\pi)^{\frac{1-k}{2k}}}
\exp\( \frac{k-1}{2k} \log(N)
%+ \frac{k-1}{2k} (sN)^{\frac{1}{k}} + O\( s^{\frac{1}{k}} + s^{\frac{2}{k}} N^{\frac{2+k}{k}}  + N^{2\delta -1} + sN^2
+O\( s^{\frac{2}{k}}N^{\frac{k+2}{k}} + sN^2 \)\).
\end{equation*}
\end{proposition}
\begin{proof}
By Lemma \ref{lem:primaryEigenvalueBound} we see that the product of the first $N$ primary eigenvalues
is
\begin{align*}
\prod_{n=1}^N \lambda_1(n) z(n)  =& \prod_{n=1}^N (ns)^{  \frac{1}{k}} \( 1+ \frac{1}{k} (ns)^{\frac{1}{k}}
+ O(ns)^{\frac{2}{k}} \) \cdot (ns)^{-1} \(1+ O(ns)\)  \\
=&  \prod_{n=1}^N (ns)^{- \frac{k-1}{k}} \(1+ \frac{1}{k} (ns)^{\frac{1}{k}} + O(ns)^{\frac{2}{k}} \)  \\
=& (N!)^{-\frac{k-1}{k}} s^{-\frac{k-1}{k}N} \exp\( \frac{s^{\frac{1}{k}} }{k+1} N^{\frac{1+k}{k}}
  + O\((sN)^{\frac{1}{k}}\) \) \\
=& (2\pi)^{-\frac{k-1}{2k}} (Ns)^{-N\frac{k-1}{k}} e^{N \(1-\frac{1}{k}\)}\\
& \hspace{.4in}
\times \exp\( - \frac{k-1}{2k} \log(N) + \frac{1}{k+1} s^{\frac{1}{k}} N^{\frac{k+1}{k}}
  + O\( (sN)^{\frac{1}{k}} \) \).
\end{align*}
%Stirling gives
%$$N!^{-\frac{k-1}{k} } = \( \frac{N}{e}\)^{-N\frac{k-1}{k}} (2\pi N)^{ - \frac{k-1}{2k}}$$
Theorem \ref{thm:runup} gives the result.
\end{proof}

%%%%%%%%%%%%%%%%%
%%     			         		%%
%%    AFTER THE RUN-UP      %%
%%						%%
%%%%%%%%%%%%%%%%%
\section{After the run-up}\label{sec:Late}

In the previous section, we computed $\wt{V}(N) = \prod_{n=1}^N m(n) {\bf e}_1$. In this section, we evaluate
$$G_k(q) =  {\bf e}_1^T \prod_{n=N}^\infty m(n)\ \wt{V}(N) $$
We have the following proposition which shows that we only need to consider the eigenvalues and the
first entry in each of the transition matrices.
%As above we assume that $\abs{\Im(s)} < \Re(s)^{\frac{3}{2}}$ and $\Re(s) >0$ throughout this section.
\begin{theorem}\label{prop:transitionapproximation}
In the notation from Lemma \ref{lem:transitionmatrix} for $N$ an integer bigger than a sufficiently large multiple of $s^{-\frac{1}{k+1}}\log(s^{-1})^{\frac{k}{k+1}}$
%and with $\abs{\Im(s)} < \Re(s)^{\frac{3}{2}}$ and $\Re(s) >0$
we have
\begin{align*}
G_k(q) =& \prod_{n=N}^\infty \lambda_1(n) z(n) \cdot \prod_{n=N}^\infty T(n)^{1,1} \cdot \wt{v_0}(N)\cdot
\(1+ O\(s+ N^{\frac{-k-1}{k}} s^{\frac{-1}{k}}\) \).
\end{align*}
\end{theorem}

In order to prove Theorem \ref{prop:transitionapproximation} we will need the following lemma.
\begin{lemma}\label{lem:Domination}
Let $w(n) := A(n)^{-1} \prod_{i=1}^{n-1} m(i) {\bf e}_1$.  Then for $n$ bigger than a sufficiently large multiple of $s^{-\frac{1}{k+1}}\log(s^{-1})^{\frac{k}{k+1}}$,
we have that for $i\neq 1$ that
$$|w(n)_i| \leq O(n^{-\frac{k+1}{k}}s^{-\frac{1}{k}} + s)|w(n)_1|.$$
\end{lemma}

\begin{proof}[Proof of Lemma \ref{lem:Domination}]
The proof is by induction on $n$.
Proposition \ref{prop:runupInEigenbasis}
makes this result clear for $n$ at the lowest end of the permissible range.  The basic idea here is that
$$w(n+1) = T(n)D(n)w(n).$$
Now since $|\lambda_1(n)| > |\lambda_i(n)|$, multiplication by $D(n)$ increases the ratio of the first entry relative to the other entries.  Since $T(n)$ is approximately $I$, multiplication by $T(n)$ does not worsen this ratio by too much.

We begin by proving our claim for $ns \ll 1$.
Letting
\begin{equation}
u(n):=D(n)w(n)
\end{equation}
 and applying Lemma \ref{lem:primaryEigenvalueBound},
we have that $$\frac{|u(n)_i|}{|u(n)_1|} \leq \frac{|w(n)_i|}{|w(n)_1|} ( 1- \Omega((ns)^{\frac{1}{k}})).$$
Next, since $T(n) = I _k+ O(n^{-1})$, and since $|u(n)_i| < k|u(n)_1|$, we have that
$$
\frac{|w(n+1)_i|}{|w(n+1)_1|} = O(n^{-1}) + \left(\frac{|w(n)_i|}{|w(n)_1|} \right)(1-\Omega((ns)^{\frac1k})).
$$
Induction on $n$ gives $$|w(n)_i| \leq O(n^{-\frac{k+1}{k}}s^{-\frac{1}{k}})|w(n)_1|$$
for all $n\ll s^{-1}$.

The argument for $ns \gg 1$ is similar.  It should be noted that in this range that $\frac{|\lambda_i(n)|}{|\lambda_1(n)|}$ is bounded above by some constant less than 1 (say by $1-\epsilon$).  Therefore, we have that
$$
\frac{|w(n+1)_i|}{|w(n+1)_1|} = O(s) + \left(\frac{|w(n)_i|}{|w(n)_1|} \right)(1-\epsilon).
$$
From this, it is easy to conclude by induction that $|w(n)_i| = O(s) |w(n)_1|$.
\end{proof}
\begin{remark}
It should be noted that the bound in Lemma \ref{lem:Domination} is not tight for small
$n$ (a stronger bound is given in Proposition \ref{prop:runupInEigenbasis}).  The bound of
$n^{-\frac{k+1}{k}}s^{-\frac{1}{k}}$ would be tight given our analysis if all
we use is that $T(n)^{1,i} = O(n^{-1})$
and that $\abs{ \frac{\lambda_i(n)}{\lambda_1(n)}} =
1-\Omega((ns)^{\frac1k})$.  In order to obtain a tighter analysis, one can
note that the $T(n)^{1,j}$ are roughly constant in $n$ and that
$\frac{\lambda_i}{\lambda_1}$ is roughly $\omega^i$, where $\omega$ is a primitive $k$th
root of unity.  By our previous analysis, $\frac{w_i(n+1)}{w_1(n+1)}$ is
approximately $\frac{\lambda_i(n)}{\lambda_1(n)}\( T(n)^{1,i} + \(\frac{w_i(n)}{w_1(n)}\) \).$
Approximating each $\frac{\lambda_i}{\lambda_1}$ by $\omega^i(1-(ns)^{\frac{1}{k}})$ and each $T(n)^{1,i}$ by a
constant of order $n^{-1}$, we note that resulting recurrence leads to terms of
size $O(n^{-1})$
%(whose absolute values differ by about $n^{-1}$)
due to cancelation that is not captured in our analysis.
\end{remark}

We are now prepared to prove Proposition \ref{prop:transitionapproximation}.
\begin{proof}[Proof of Theorem \ref{prop:transitionapproximation}]
We claim that
$$
w(n+1)_1 = w(n)_1 \lambda_1(n)z(n)T(n)^{1,1}(1+O(\min(n^{-\frac{2k+1}{k}}s^{-\frac{1}{k}},s^2 z(n)))).
$$
Or equivalently (since $u(n)_1 = \lambda_1(n)z(n) w(n)_1$) that
$$
w(n+1)_1 = u(n)_1 T(n)^{1,1}(1+O(\min(n^{-\frac{2k+1}{k}}s^{-\frac{1}{k}} , s^2 z(n)))).
$$
It is clear that
$$
w(n+1)_1 = \sum_j T(n)^{1,j} u(n)_j
$$
Hence we need to show
$$
\max_{j\neq 1} \( T(n)^{1,j} \cdot \frac{ |u(n)_j|}{|u(n)_1|}\) = O(\min(n^{-\frac{2k+1}{k}}s^{-\frac{1}{k}} + s, s^2 z(n))).
$$
If $ns\ll 1$, this follows since $T(n)^{1,j} \ll n^{-1}$, and
$\frac{|u(n)_j|}{|u(n)_1|} \leq \frac{|w(n)_j|}{|w(n)_1|} = O(n^{-\frac{k+1}{k}}s^{-1}).$  Otherwise, this follows from noting that
$T(n)^{1,j} \ll s$ and $$\frac{|u(n)_j|}{|u(n)_1|} = \left( \frac{|\lambda_j(n)|}{|\lambda_1(n)|}\right)
\left(\frac{ |w(n)_j|}{|w(n)_1|} \right) = O(z(n) s).$$
 This proves the claim.

Therefore we have that
$$
\lim_{n\rightarrow\infty} w(n)_1 = \prod_{n=N+1}^\infty \lambda(n)z(n)T(n)^{1,1} \cdot
\exp\left( O\left( \sum_{n=N+1}^\infty \min(n^{-\frac{2k+1}{k}}s^{-\frac1k}, s^2 z(n))\right) \right).
$$
The sum in the error term is at most
$$
\sum_{n=N+1}^{\lfloor  s^{-1}\rfloor } n^{-\frac{2k+1}{k}} s^{-\frac{1}{k}} + \sum_{n= \lfloor s^{-1} \rfloor}^\infty s^2 z(n).
$$
The first term is
$
O\(N^{-\frac{k+1}{k}}s^{-\frac{1}{k}}\)
$
and the latter term is
%above is at most
$
O\left(s^2\sum_{n=1}^\infty e^{-ns}\right) = O(s).
$
\end{proof}

The following theorem is enough to deduce
Theorem \ref{thm:partitionAsymp} and thus Theorem \ref{thm:main}
%it suffices to
%have an approximation for $\prod_{n=N}^\infty T_{n}^{1,1}$.
\begin{theorem}\label{thm:transitionProduct}
With $N$ as above we have
$$\prod_{n=N}^\infty T(n)^{1,1}  = k^{\frac{1}{2}}  \exp\( -\frac{k-1}{2k} \log\(Ns\)
%+ \frac{k-1}{2k} (Ns)^{\frac{1}{k}}
+ O\( (Ns)^{\frac{1}{k}} + N^{-1} + s\)\).$$
\end{theorem}
\begin{proof} Throughout this proof we use the notation of Lemma \ref{lem:transitionmatrix} and often suppress the dependence
on $n$.
We have
$$T(n-1)^{1,1} = \prod_{m\ne 1} \frac{\mu_1 - \lambda_m}{\lambda_1 - \lambda_m}  \cdot \frac{\lambda_1}{\mu_1}$$
and
$$\mu_1(n) = \lambda_1(n-1)  = \lambda_1(n) - \lambda_1'(n) + O( \lambda_1''(n)) $$
where $\lambda_1'(n) = \frac{\partial }{\partial n} \lambda_1(n)
%= \frac{\partial}{\partial z} \lambda_1(z) \cdot \( - s \frac{e^{-ns}}{(1-e^{-ns})^2}\)
$.
Therefore,
$$\frac{\mu_1 - \lambda_m}{\lambda_1 - \lambda_m}  \cdot \frac{\lambda_1}{\mu_1}
= 1+  \lambda_1' \( \frac{1}{\lambda_1 - \lambda_m} - \frac{1}{\lambda_1}\)
+  O_k\( \( \frac{\lambda_1''}{\lambda_1} + \(\frac{\lambda_1'}{\lambda_1}\)^2\)\)$$
Hence,
$$T(n-1)^{1,1} = \exp\(  -\lambda_1' \sum_{m\ne 1}\( \frac{1}{\lambda_1 - \lambda_m} - \frac{1}{\lambda_1}\) + O_k\(
  \( \frac{\lambda_1''}{\lambda_1} + \(\frac{\lambda_1'}{\lambda_1}\)^2\)\)\).$$
To estimate the big-$O$ term for $ns\ll 1$ we use  \eqref{eqn:firstderivative} and \eqref{eqn:secondderivative} and
Lemma \ref{lem:primaryEigenvalueBound} to obtain
\begin{align*}
\frac{1}{\lambda_1} \frac{\partial \lambda_1}{\partial n}  =& -s \frac{1}{\lambda_1} \frac{\partial \lambda_1}{\partial z} \cdot z(n)^2e^{ns}  =
O\(\frac{1}{n} \) \\
\frac{1}{\lambda_1} \frac{\partial^2 \lambda_1}{\partial n^2}  =& \frac{s^2e^{2ns}}{\lambda_1}
\( \frac{\partial^2 \lambda_1 }{\partial z^2} \cdot z(n)^4 + \frac{\partial \lambda_1}{\partial z} \cdot  z(n)^3 \)
= O\( \frac{1}{n^2}\).
\end{align*}
For $ns\gg 1$ we use Lemma \ref{lem:bound2} to obtain
\begin{align*}
\frac{1}{\lambda_1} \frac{\partial \lambda_1}{\partial n}
%=& -s \frac{\partial}{\partial z}\log\( \lambda_1(z)\) \cdot z(n) e^{ns}
 =  O\(se^{-ns}\) \ \ \text{ and } \ \
\frac{1}{\lambda_1} \frac{\partial^2 \lambda_1}{\partial n^2}
%=& -s^2 \( \frac{\partial \lambda_1 }{\partial z^2} + \frac{\partial \lambda_1}{\partial z}\)  z(n)^2 e^{ns}
 = O\( s^2 e^{-ns}\).
\end{align*}
Therefore
\begin{align*}
\prod_{n=N}^\infty T(n-1)^{1,1} \exp\(\lambda_1' \sum_{m\ne 1} \( \frac{1}{\lambda_1 - \lambda_m} - \frac{1}{\lambda_1}\)\)
=& \exp\( \sum_{n=N}^{\lfloor \frac{1}{s}\rfloor} O\(\frac{1}{n^2}\) + O\( s^2 \sum_{n=\lfloor \frac{1}{s}\rfloor}^\infty e^{-ns}\)\) \\
=& \exp\( O\(\frac{1}{N} + s\)\).
\end{align*}

Let
$P(\lambda, z):= \lambda^k - z^{-1} \( \lambda^{k-1} + \cdots + \lambda+1\).$
 We have
\begin{equation}\label{eqn:transitionIntegral}
2\sum_{m\ne 1} \frac{1}{\lambda_1 - \lambda_m} = \frac{\frac{\partial^2}{\partial \lambda^2} P(\lambda, z)}{\frac{\partial}{\partial \lambda} P(\lambda, z)} \big |_{\begin{subarray}{c} \lambda  = \lambda_1 \\ z = z(N) \end{subarray}}
=: R_k(\lambda_1).
\end{equation}
Therefore,
%$$T(n)^{1,1} = \exp\( s \lambda_1'(n)  R_k(\lambda_1(n)) - s(k-1) \frac{\lambda_1'(n)}{\lambda_1(n)} +
%+ O_k\( \frac{1}{N}  + s\)\).$$
%Thus
\begin{align*}
\prod_{n=N}^\infty T(n)^{1,1}
=& \prod_{n=N}^\infty T(n-1)^{1,1}(1+O(N^{-1}+s)) \\
&= \exp\( - \sum_{n=N}^\infty \(\frac{1}{2} \lambda_1'(n) R_k(\lambda_1(n)) - (k-1) \frac{\lambda_1'(n)}{\lambda_1(n)}\)
+ O\( N^{-1} + s \) \)
\end{align*}
We apply Euler-MacLaurin to approximate the sum by an integral.
The error from the terms $\frac{\lambda_1'(n)}{\lambda_1(n)}$ introduces an error of size
$$\int_N^{\infty} \( \frac{\lambda_1''(n)}{\lambda_1(n)} + \(\frac{\lambda_1'(n)}{\lambda_1(n)}\)^2 \) dn = O\(N^{-1} + s\)$$
 as above.
% A similar argument gives the same result for the other approximation.
Thus, we have
\begin{align*}
\prod_{n=N}^\infty T(n)^{1,1}
=&\exp\( -\int_{N}^\infty \( \frac{1}{2} \lambda_1'(x) R_k(\lambda_1(x)) - (k-1) \frac{\lambda_1'(x)}{\lambda_1(x)}\)dx
  + O\(N^{-1} + s\) \)\\
=& \exp\( - \int_{\lambda_1(N)}^\infty \frac{R_k(x)}{2} - \frac{k-1}{x}dx + O\( N^{-1} + s\)\)
%\\=& \exp\(\frac{1}{2} \log\( \frac{\partial}{\partial \lambda} P(\lambda, z(N)) \Big |_{\lambda_1(N)} \) - (k-1) \log(\lambda_1(N)) + O\( \frac{1}{N} + s\) \) \\
\end{align*}
\begin{comment}
where we have used \eqref{eqn:transitionIntegral} and the fact that $R_k$ is the logarithmic derivative of
$\frac{\partial}{\partial \lambda} P(\lambda, z)$.
Now we use
\begin{align*}
\frac{\partial}{\partial \lambda} P(\lambda, z)\Big |_{\lambda = \lambda_1}
=& k \lambda^{k-1} - z^{-1} \( (k-1) \lambda^{k-2} + \cdots + 2\lambda + 1\)\Big |_{\lambda = \lambda_1}
\\ =& k \lambda_1^{k-1} \(1 - O\( z^{-1} \frac{\lambda_1^{k-1} - 1}{\lambda_1 (\lambda_1 - 1)} \)\).
\end{align*}
\end{comment}
In order to evaluate the integral
$
\int R_k(x)dx,
$
we let $a(\lambda)=\lambda^k$ and $b(\lambda)=\lambda^{k-1}+\cdots+1$.  We then have that $z^{-1}=\frac{a(\lambda_1)}{b(\lambda_1)}.$  Therefore,
$$
R_k(\lambda) = \frac{a''(\lambda)-z^{-1}b''(\lambda)}{a'(\lambda)-z^{-1}b'(\lambda)} = \frac{a''(\lambda)b(\lambda)-a(\lambda)b''(\lambda)}{a'(\lambda)b(\lambda)-a(\lambda)b'(\lambda)} = \frac{\partial}{\partial \lambda}\log( a'(\lambda)b(\lambda)-a(\lambda)b'(\lambda) ).
$$
Letting
\begin{align*}
Q(\lambda) & := a'(\lambda)b(\lambda)-a(\lambda)b'(\lambda)\\ & = k\lambda^{k-1}(\lambda^{k-1}+\cdots+1) - \lambda^{k}((k-1)\lambda^{k-2}+\cdots+1)\\ & = \lambda^{2k-2} +2\lambda^{2k-3} + \cdots + k\lambda^{k-1},
\end{align*}
we have that
\begin{align*}
\int_{\lambda_1(N)}^\infty \frac{R_k(x)}{2} - \frac{k-1}{x}dx & = \frac{1}{2}\left[ \log\left(Q(\lambda)\lambda^{-2k+2} \right)\right]_{\lambda_1(N)}^\infty.
\end{align*}
We note that for $\lambda \gg 1$ that $Q(\lambda)\lambda^{-2k+2}=1+O(\lambda^{-1})$, and therefore, $$\lim_{\lambda\rightarrow\infty}\log\left(Q(\lambda)\lambda^{-2k+2} \right)=0.$$
For $\lambda \ll 1$, we have that $Q(\lambda)\lambda^{-2k+2}=k\lambda^{-k+1}(1+O(\lambda)).$
Therefore
\begin{align*}
\prod_{n=N}^\infty T(n)^{1,1} &= \exp\left(\frac{1}{2} \log\left(k\lambda_1(N)^{-k+1}(1+O(\lambda_1(N))) \right) + O\( N^{-1} + s\) \right).
\end{align*}
By Lemma \ref{lem:primaryEigenvalueBound}, we have
%$$ \frac{\partial}{\partial \lambda} P(\lambda, z) \Big |_{\lambda = \lambda_1(N) } = k (Ns)^{\frac{k-1}{k}}\( 1 + \frac{k-1}{k} (Ns)^{\frac{1}{k}} + O\(Ns\)^{\frac{2}{k}} \).$$
%Applying Lemma \ref{lem:primaryEigenvalueBound} a second time we have
\begin{align*}
\prod_{n=N}^\infty T(n)^{1,1}  =&
\exp\( -\frac{k-1}{2} \log\(\lambda_1(N)\) + \frac{1}{2} \log(k)  +  O\( (Ns)^{\frac{1}{k}}  + N^{-1} + s\)\) \\
=& \exp\( -\frac{k-1}{2k} \log\(Ns\) + \frac{1}{2} \log(k)  - \frac{k-1}{2k} (Ns)^{\frac{1}{k}} +  O\( (Ns)^{\frac{1}{k}}  + N^{-1} + s\)\).
\end{align*}
\end{proof}

In the next section, we analyze the product of the primary eigenvalues.
%%%%%%%%%%%%%%%%%%
%%                                               	   %%
%%   EIGENVALUE PRODUCT  %%
%%						   %%
%%%%%%%%%%%%%%%%%%
\section{The Product of the Primary Eigenvalues}\label{sec:Eigenvalues}
In this section, we estimate
$$\prod_{n=1}^\infty \lambda_1(n) z(n)  =
\exp\( \sum_{n=1}^\infty  \log(\lambda_1(n)) + \log \( \frac{q^n}{1-q^n}\) \).$$
%Throughout this section we will assume $\abs{\Im(s)} < \Re(s)^{\frac{3}{2}}$ and $\Re(s) >0$.

\begin{theorem}\label{thm:eigenvalueProduct}
In the notation above we have
\begin{align*}
 \sum_{n=1}^\infty  \log\(\lambda_1(n) z(n)\)   =&  \frac{\pi^2}{6s} \( 1- \frac{2}{k(k+1)}\) +
\(\frac{k-1}{2k}\) \log(s) - \(\frac{k-1}{2k}\) \log(2\pi)
%\\&\hspace{.3in}  +\frac{k+1}{2^{\frac{k+1}{k}} k^2} s^{\frac{1}{k}} +  O\(s^{\frac{2}{k}}\).
 +  O_k\(s^{\frac{1}{k}}\).
\end{align*}
\end{theorem}

We start with the following lemma which closely resembles Euler-MacLaurin summation.
\begin{lemma}\label{lem:ml}
For suitable functions $h$  and $n\ge 1$ we have
\begin{align*}
h(n) =& \int_{n-\frac{1}{2}}^{n+\frac{1}{2}} h(z) dz - \int_{n-\frac{1}{2}}^{n+\frac{1}{2}} h'(x) \( [x] - x + \frac{1}{2}\) dx\\
=&  \int_{n-\frac{1}{2}}^{n+\frac{1}{2}} h(z) dz -\frac{1}{2} \int_{n-\frac{1}{2}}^{n+\frac{1}{2}}  h''(x) \( [x]-x+\frac{1}{2}\)^2 dx
\end{align*}
%\begin{align*}
%\sum_{n=1}^\infty h(n) =&  \int_0^\infty h(x) dx - \int_0^\frac{1}{2} h(x) dx -
%\int_{\frac{1}{2}}^\infty h'(x)\( [x] - x + \frac{1}{2}\) dx
%\end{align*}
where $[x]$ denotes the integer part of $x$.
\end{lemma}
\begin{proof}
To see this note that
for any function $h(z)$ we have
$$h(z) = h(n) + h'(n) (z-n) + \int_n^z h''(x) (z-x) dx.$$
Integrating from $n-\frac{1}{2}$ to $n+ \frac{1}{2}$ gives the second result.
Integration by parts on each interval $[n,n+\frac{1}{2}]$ and $[n-1/2, n]$ gives
the result first result.
\end{proof}

Define the function $f_k(e^{-x})$ to be the increasing function satisfying
\begin{equation}
f_k(e^{-x})^{k+1} - f_k(e^{-x})^{k} = e^{-x(k+1)} - e^{-xk}.
\end{equation}
Since $\lambda_1(n)^k = z(n)^{-1} \( \lambda_1(n)^{k-1} + \cdots \lambda_1(n) + 1\)$, multiplying by $\lambda_1(n)-1$ we have
$\lambda_1(n)^{k+1} - \lambda_1(n)^k = z(n)^{-1} (\lambda_1(n)^k-1) = q^{-n} \lambda_1^k - q^{-n} - \lambda_1^k + 1$.
Therefore $f_k(e^{-ns}) = \lambda_1(n) q^n$.
\begin{remark}
This function $f_k(e^{-x})$, and certain generalizations, are studied  in \cite{HLR}.
\end{remark}

\begin{proof}[Proof of Theorem \ref{thm:eigenvalueProduct}]
The modularity of the Dedekind $\eta$-function gives
\begin{equation}\label{eqn:etaasymptotic}
\sum_{n=1}^\infty \log\(1-q^n\) = \frac{\pi^2}{6s} +\frac{1}{2} \log(s) - \frac{1}{2} \log(2\pi) - \frac{s}{24} + O(s^{M})
\end{equation}
for any $M>0$.
Additionally, by Lemma \ref{lem:ml}, we have
$$\sum_{n=1}^\infty \log\(1-q^n\)  = \int_0^\infty \log(1-e^{-xs}) dx - \int_0^\frac{1}{2} \log(1-e^{-xs}) dx
-s \int_\frac{1}{2}^\infty \frac{e^{-xs}}{1-e^{-xs}}\( [x] - x + \frac{1}{2}\)  dx.$$
Noting that $\int_0^\infty \log(1-e^{-xs}) dx  = \frac{\pi^2}{6s}$ and
$\int_0^\frac{1}{2} \log(1-e^{-xs}) dx =  \frac{1}{2} \log(s) + \int_0^\frac{1}{2} \log(x) dx + O(s)$
we may conclude that
\begin{equation}\label{eqn:hack}
-\int_0^\frac{1}{2} \log(x) dx  - s \int_{\frac{1}{2}}^\infty \frac{e^{-xs}}{1-e^{-xs}} \( [x] - x + \frac{1}{2}\)  dx =
\frac{1}{2}\log(2\pi) + O(s).\end{equation}

Following the notation of Section 3 of \cite{BM} we define
\begin{equation}
g_k(xs) = - \log(f_k(e^{-xs})).
\end{equation}
By Lemma \ref{lem:ml},
\begin{equation}\label{eqn:afterml}
\sum_{n=1}^\infty g_k(ns) = \int_{0}^\infty g_k(xs) dx - \int_0^{\frac{1}{2}} g_k(xs) dx
- s \int_{\frac{1}{2}}^\infty g_k'(xs) \([x] - x + \frac{1}{2}\)dx.
\end{equation}
Theorem 1 of \cite{HLR} gives
$\int_0^\infty g_k(xs) dx = \frac{1}{s} \frac{\pi^2}{3k(k+1)}$.
%\begin{remark}
%This integral evaluation
%can also be seen to hold for any $s$ of small argument by changing the line of integration so that $xs$ is real and then applying the theorem of \cite{HLR}.
%\end{remark}
Lemma \ref{lem:primaryEigenvalueBound}
gives that for $sx\ll 1$
$$
g_k(xs) = -\log\(f_k(e^{-xs})\) = -\frac{1}{k}\log(xs)  + \frac{1}{k} (xs)^{\frac{1}{k}} + O\((xs)^{\frac{2}{k}}\)
$$
Therefore,
we have
\begin{equation}\label{eqn:1/2integral}
-\int_0^\frac{1}{2} g_k(xs) dx= \frac{1}{2k} \log(s) -\frac{1}{k} \int_0^\frac{1}{2} \log(x) dx
%+ \frac{k+1}{2^{\frac{k+1}{k}} k^2}s^{\frac{1}{k}}
+O\(s^{\frac{1}{k}}\).
\end{equation}

Let $M = \lfloor s^{-\frac{1}{k} } \rfloor$. Then we have
\begin{align}
s\int_{M+\frac{1}{2}}^\infty g_k'(xs) \([x]-x+\frac{1}{2}\) dx
 =&\frac{s^2}{2}\int_{M+\frac{1}{2}}^{\infty}g_k''(xs)  \([x]-x+\frac{1}{2}\)^2 dx  \nonumber \\
 \ll&  s\int_{Ms}^\infty g''(w) dw \ll M^{-1} \ll s^{\frac{1}{k}} \label{eqn:ml_estimate1}
\end{align}
where we use $g'(Ms) = O_k\(\frac{1}{Ms}\)$ (see, for instance, Lemma 3.1 of \cite{BM}).

To estimate the integral of $g_k'$ from $\frac12$ to $M+\frac{1}{2}$ we
 take the logarithmic derivative of $f_k(e^{-w})^{k+1} - f_k(e^{-w})^k = e^{-w(k+1)} - e^{-wk}$
to obtain
$$g_k'(w) = 1 - \frac{1}{k} \frac{e^{-w}}{e^{-w}-1} + \frac{1}{k} e^{-w} \frac{f_k'(e^{-w})}{1- f_k(e^{-w})}.$$
Therefore
\begin{align}
s\int_\frac{1}{2}^{M+\frac{1}{2}} g_k'(xs) \([x]-x+\frac{1}{2}\) dx  = &
- \frac{s}{k} \int_{\frac{1}{2}}^{M+\frac{1}{2}} \frac{e^{-xs}}{1-e^{-xs}}\([x]-x+\frac{1}{2}\) dx  \nonumber
 \\ &\hspace{.4in} + \frac{s}{k}\int_{\frac{1}{2}}^{M+\frac{1}{2}} e^{-xs}
 \frac{f_k'(e^{-xs})}{1-f_k(e^{-xs})} \([x]-x+\frac{1}{2}\)  dx  \label{eqn:ml_estimate2a}
\end{align}
Observe that we have
\begin{align}
 \int_{\frac{1}{2}}^{M+\frac{1}{2}}& \frac{e^{-xs}}{1-e^{-xs}}\([x]-x+\frac{1}{2}\) dx \nonumber \\=&
   \int_{\frac{1}{2}}^\infty \frac{e^{-xs}}{1-e^{-xs}}\([x]-x+\frac{1}{2}\) dx - \int_{M+\frac{1}{2}}^{\infty}  \frac{e^{-xs}}{1-e^{-xs}}\([x]-x+\frac{1}{2}\) dx \nonumber \\
   =&    \int_{\frac{1}{2}}^\infty \frac{e^{-xs}}{1-e^{-xs}}\([x]-x+\frac{1}{2}\) dx + \frac{s}{2}
   \int_{M+\frac{1}{2}}^{\infty}  \frac{e^{-xs}}{(1-e^{-xs})^2}\([x]-x+\frac{1}{2}\)^2 dx  \nonumber \\
   =&   \int_{\frac{1}{2}}^\infty \frac{e^{-xs}}{1-e^{-xs}}\([x]-x+\frac{1}{2}\) dx  + O(se^{-Ms}). \label{eqn:ml_estimate2}
\end{align}

Additionally, integrating by parts we obtain
\begin{align}
s\int_{\frac{1}{2}}^{M+\frac{1}{2}}& e^{-xs}
 \frac{f_k'(e^{-xs})}{1-f_k(e^{-xs})} \([x]-x+\frac{1}{2}\)  dx \nonumber \\
%   \ll & \abs{s} \int_{\frac{1}{2}}^{N+\frac{1}{2}}  \frac{\partial}{\partial w} \(e^{-ws}
%    \frac{f_k'(e^{-ws})}{1-f_k(e^{-ws})}\) \Big |_{w=x} \([x]-x+\frac{1}{2}\)^2  dx \nonumber \\
\ll &   s \cdot \frac{e^{-xs} f_k'(e^{-xs})}{1-f_k(e^{-xs}) }\Big |_{\frac{1}{2}}^{M+\frac{1}{2}}
% = \log\( 1- \(\frac{s}{2}\)^{\frac{1}{k}}\) + \log\( 1- (Ns)^{\frac{1}{k}}\)
 \ll_k s^{\frac{1}{k}} \(1+ M^{-\frac{k-1}{k}}\)\label{eqn:ml_estimate3}
\end{align}
where we have used that monotonicity of $\log(1-f_k(w))$ and $f_k'(z) = O\(z^{\frac{1-k}{k}}\)$ for $z$ near $0$.
Returning to \eqref{eqn:afterml} and using \eqref{eqn:hack} and \eqref{eqn:ml_estimate1}-\eqref{eqn:ml_estimate3}
\begin{align*}
 -\frac{1}{k} &\int_0^\frac{1}{2} \log(x) dx - s \int_{\frac{1}{2}}^\infty g_k'(xs) \([x] - x + \frac{1}{2} \) dx  \\
  =&
 \frac{1}{k}\(  -\int_0^\frac{1}{2} \log(x) dx - s \int_{\frac{1}{2}}^\infty \frac{e^{-xs}}{1-e^{-xs}}\([x]-x+\frac{1}{2}\) dx  \) +
 O\(s^{\frac{1}{k}} + M^{-1}\)  \\
 =& \frac{1}{2k}\log(2\pi)  + O(s^{\frac{1}{k}}) \end{align*}
Finally, this together with \eqref{eqn:afterml} and \eqref{eqn:1/2integral} gives the result.
\end{proof}

 %%%%%%%%%%%%%%%%%%
%%                                               	   %%
%%  	    NON-REAL S                     %%
%%						   %%
%%%%%%%%%%%%%%%%%%
\section{Proof of Theorem \ref{thm:partitionAsymp}}\label{sec:NonReal}
In this section, we prove Theorem \ref{thm:partitionAsymp} and thus Theorem \ref{thm:main}.

\begin{proof}[Proof of Theorem \ref{thm:partitionAsymp}]
We have $G_k(q) = {\bf e}_1^T \prod_{n=N+1}^\infty m(n) \cdot \prod_{n=1}^N m(n) {\bf e}_1$.
It follows from Theorem  \ref{prop:transitionapproximation},
Proposition \ref{prop:transitionComparison},
 and Theorems  \ref{thm:transitionProduct} and \ref{thm:eigenvalueProduct} that for appropriate $N$,
\begin{align*}\label{eqn:GkTransition}
G_k(e^{-s})
=&
\frac{1}{k} \exp\(\frac{\pi^2}{6s} \(1-\frac{2}{k(k+1)}\)
%\right. \\ & \hspace{.8in}\left.
+ O\(
%s^{\frac{2}{k}}N^{\frac{k+2}{k}}  + N^{2\delta - 1}  +
N^{-\frac{k+1}{k}} s^{-\frac{1}{k}} + sN^2 +s^{\frac{2}{k}}N^{\frac{k+2}{k}}  + N^{-1}\)\).
\end{align*}
Setting $N=\floor{s^{-\frac{3}{2k+3}}}$ yields
the result.
\end{proof}

 %%%%%%%%%%%%%%%%%%
%%                                               	   %%
%%  	        CIRCLE METHOD        %%
%%						   %%
%%%%%%%%%%%%%%%%%%
\section{Proof of Theorem \ref{thm:ksequencePartitions}}\label{sec:PartitionAsymp}
In this section we apply a result of Ingham \cite{ingham} to deduce the asymptotics for $p_k(n)$
from the asymptotics of $G_k(q)$ as $q \to 1$.  In particular, we have the following result
which is a special case of Theorem 1 of \cite{ingham} and is given as Theorem 4.1 of \cite{BHMV}.
\begin{theorem}[Ingham]\label{thm:ingham} Let $f(z) = \sum_{n=0}^\infty a(n) z^n$ be a power series with real
nonnegative coefficients and radius of convergence equal to 1. If there exists $A>0$, $\lambda$,
$\alpha \in \R$ such that
$$f(z) \sim \lambda (-\log(z))^\alpha \exp\( - \frac{A}{\log(z)}\)$$ as $z\to 1^-$, then
$$\sum_{m=0}^n a(m) \sim \frac{\lambda}{2\sqrt{\pi}} \frac{A^{\frac{\alpha}{2} - \frac{1}{4}}}{n^{\frac{\alpha}{2} + \frac{1}{4}}}
\exp\( 2\sqrt{An}\)$$
as $n\to \infty$.
\end{theorem}
\begin{proof}[Proof of Theorem \ref{thm:ksequencePartitions}]
By Lemma 10 of \cite{HLR} $(1-q)G_k(q) = \sum_{n=0}^\infty (p_k(n) - p_k(n-1)) q^n$
has nonnegative coefficients.  Applying Theorems \ref{thm:partitionAsymp} and \ref{thm:ingham} gives the result.
\end{proof}

\end{document}